\documentclass[11pt]{amsart}
\newtheorem{theorem}{Theorem}[section]
\newtheorem*{thmA}{Theorem}
\newtheorem{lemma}[theorem]{Lemma}

\theoremstyle{definition}

\newcommand{\cu}{\mathcal{K}}
\newcommand{\ca}{{\mathcal{K}}_0}

\theoremstyle{remark}

\numberwithin{equation}{section}

\begin{document}

\title[Volume preserving centro-affine normal flows]
 {Volume preserving centro-affine normal flows}

\author[M. N. Ivaki]{Mohammad N. Ivaki}
\address{Department of Mathematics and Statistics,
  Concordia University, Montreal, QC, Canada, H3G 1M8}
\email{mivaki@mathstat.concordia.ca}

\author[A. Stancu]{Alina Stancu}
\address{Department of Mathematics and Statistics,
  Concordia University, Montreal, QC, Canada, H3G 1M8}
\email{stancu@mathstat.concordia.ca}

\dedicatory{}

\begin{abstract}
We study the long time behavior of the volume preserving $p$-flow in $\mathbb{R}^{n+1}$ for $1\leq p<\frac{n+1}{n-1}$. By extending Andrews' technique for the flow along the affine normal, we prove that every centrally symmetric solution to the volume preserving $p$-flow converges sequentially to the unit ball in the $C^{\infty}$ topology, modulo the group of special linear transformations.
\end{abstract}

\maketitle

\section{Introduction}
Let $K$ be a compact, centrally symmetric, strictly convex body, smoothly embedded in $\mathbb{R}^{n+1}$.
We denote the space of such convex bodies by $\mathcal{K}_{sym}$. Let
 $$x_K:\mathbb{S}^n\to\mathbb{R}^{n+1}$$ be
 the Gauss parametrization of $\partial K$, the boundary of $K\in \mathcal{K}_{sym}$, where the origin of the Euclidean space is chosen to coincide with the center of
 symmetry of the body. The support function of $\partial K$ is defined by
 $$s_{\partial K}(z):= \langle x_K(z), z \rangle,$$
 for each $z\in\mathbb{S}^n$.  We denote the matrix of the radii of curvature of $\partial K$ by
$\mathfrak{r}=[\mathfrak{r}_{ij}]_{1\leq i,j\leq n}$, the entries of $\mathfrak{r}$  considered as functions on the unit sphere too. They are related to the support function by the identity
 $$\mathfrak{r}_{ij}:=\bar{\nabla}_i\bar{\nabla}_j s+s\bar{g}_{ij},$$
 where $\bar{g}_{ij}$ is the standard metric on $\mathbb{S}^{n}$ and $\bar{\nabla}$ is the standard Levi-Civita connection of $\mathbb{S}^{n}.$
 We denote the Gauss curvature of $\partial K$ by $\mathcal{K}$ and remark that, as a function on the unit sphere, it is related to the support function of the convex body by $$\frac{1}{\mathcal{K}}:=S_n=\det_{\bar{g}}(\bar{\nabla}_i\bar{\nabla}_js+\bar{g}_{ij}s):=\frac{\det \mathfrak{r}_{ij}}{\det{\bar{g}_{ij}}}.$$
 We denote the eigenvalues of $[\mathfrak{r}_{ij}]_{1\leq i,j\leq n}$ with respect to the metric $\bar{g}_{ij}$ by $\lambda_i$ for $1\leq i\leq n.$ Then the principal curvatures of $\partial K$ are $\kappa_i=\frac{1}{\lambda_i}$ for $1\leq i\leq n$.

For any fixed convex body $K_0$ whose interior contains the origin, and whose boundary is of class $C^2$ with strictly positive Gauss curvature, the following flow along the centro-affine normal has been defined in \cite{S}:
 $$\frac{\partial}{\partial t}x:=-\ca^{\alpha} \mathcal{N}_0,~~
 x(\cdot,0)=x_{K_0}(\cdot),
$$
where $\alpha$ is a positive real power which will be made explicit soon, $\ca=\cu/s^{n+2}$ is, interestingly, an $SL(n+1)$ invariant called centro-affine curvature and, finally, $ \mathcal{N}_0= {\ca}^{-\frac{1}{n+1}}(z)(- {\cu}^{\frac{1}{n+1}}(z)\, z-\bar{\nabla} ({\cu}^{\frac{1}{n+1}}(z)))$ is the centro-affine normal, both as functions of $z$. The flow was defined for the purpose of finding new global centro-affine invariants of smooth convex bodies in which a certain class of existing invariants arose naturally. Only the short time existence to the flow was then needed. Moreover, several interesting isoperimetric type inequalities  were obtained via short time existence of the flow, \cite{S}.
 In what follows, we will consider an equivalent (Euclidean) formulation of the flow and, in this paper, we will restrict the power to a certain range.

Let $p \geq 1$ be a fixed real number and let $K_0\in \mathcal{K}_{sym}$. We consider a family of convex bodies $\{K_t\}_t\in \mathcal{K}_{sym}$, and their  associated smooth
 embeddings $x:\mathbb{S}^n\times[0,T)\to \mathbb{R}^{n+1}$, which are evolving according  to the $p$-centro-affine flow, namely,
 \begin{equation}\label{e: p flow ev of x}
 \frac{\partial}{\partial t}x:=-s\left(\frac{\mathcal{K}}{s^{n+2}}\right)^{\frac{p}{n+1+p}}\, z,~~
 x(\cdot,0)=x_{K_0}(\cdot),~~ x(\cdot ,t)=x_{K_t}(\cdot).
 \end{equation}
  It is this flow for which we propose the study of asymptotic behavior by applying the techniques of \cite{BA1}.
  The long time behavior of the flow in $\mathbb{R}^2$ was studied by the first author in \cite{Ivaki} using tools of affine differential geometry. It was proved there that the volume preserving $p$-flow with $p \geq 1$ evolves any convex body in $ \mathcal{K}_{sym}$ to the unit disk in Hausdorff distance, modulo $SL(2).$ A further application of the techniques developed in \cite{Ivaki} to the $L_{-2}$ Minkowski problem is given in \cite{Ivaki2}. The case $p=1$, the well-known affine normal flow, was already addressed by Andrews \cite{andrews}, \cite{BA1}. Andrews,
 investigated the affine normal flow of compact, convex hypersurfaces in any dimension and showed that the volume preserving flow evolves any convex initial bounded open set exponentially fast, in the $C^{\infty}$ topology, to an ellipsoid. In another direction, interesting results for the affine normal flow have been obtained in \cite{LT} by Loftin and Tsui regarding ancient solutions, and existence and regularity of solutions on non-compact strictly convex hypersurfaces.

In this paper, we prove that
\begin{thmA}[Main Theorem]
Let $1\leq p<\frac{n+1}{n-1}$ be a real number. Let $x_{K_0}:\mathbb{S}^n\to\mathbb{R}^{n+1}$ be a smooth, strictly convex embedding of $K_0\in\mathcal{K}_{sym}.$ Then there exists a unique solution $x:\mathbb{S}^n\times [0,T)\to\mathbb{R}^{n+1}$ of equation (\ref{e: p flow ev of x}) with initial data $x_{K_0}$. The rescaled convex bodies given by the embeddings $\left(\frac{V(B^{n+1})}{V(K_t)}\right)^{\frac{1}{n+1}}x_{K_t}$ converge sequentially in the $C^{\infty}$ topology to the unit ball, modulo $SL(n+1).$
\end{thmA}

The restriction on the range of $p$ is purely technical and it will become evident in Lemma \ref{lem: lower G} which is the only place requiring it.

\section{Uniform lower and upper bound on the principal curvatures}
We will start this section by proving that, under uniform lower and upper bounds on the support function of the evolving convex body, we have uniform lower and upper bounds on the Gauss curvature which depend only on the dimension $n$, the value of $p$, the bounds on the support function, and time. To obtain the upper bound on the Gauss curvature, we apply a standard technique of Tso \cite{Tso}. To derive the lower bound bound on the Gauss curvature, we consider  the evolution of the dual convex body and we apply Tso's technique to the speed of the dual $p$-flow. This procedure avoids the need for a Harnack estimate, or displacement bounds.

It is easy to see from $\langle x(z), z \rangle = s(z)$
 that as the convex bodies $\{K_t\}_t$ evolve by (\ref{e: p flow ev of x}) their support functions satisfy the partial differential equation
 \begin{equation}\label{e: p flow ev of s}
\partial_t s=-s\left(\frac{\mathcal{K}}{s^{n+2}}\right)^{\frac{p}{n+1+p}},
\end{equation} see also \cite{S}. The short time existence and uniqueness of solutions for a smooth and strictly convex initial hypersurface follow from the strict parabolicity of the equation and it was shown in \cite{S}.
We will use this latter evolution equation to describe the flow throughout the rest of the paper.

The proofs of the two lemmas pertaining to upper, respectively, lower bound of the Gauss curvature of the evolving convex bodies have similar outline yet with some differences. For completeness, we will present both proofs.

 Given a convex body $K$, the inner radius of $K$, $r_-(K)$, is the radius of the largest ball contained in $K$; the outer radius of $K$, $r_+(K)$, is the radius of the smallest ball containing $K$. Notice that for any centrally symmetric convex body, the smallest and largest ball as above will be centered at the origin.

\begin{lemma}[Upper bound on the Gauss curvature]\label{lem: upper G}
For any smooth, strictly convex solution $\{K_t\}_{[0,t_0]}$ of the evolution equation (\ref{e: p flow ev of s}) with $0<R_{-}\leq r_{-}(K_t)\leq r_{+}(K_t)\leq R_{+} < +\infty$ for $t\in[0,t_0]$, and some positive numbers $R_{\pm}$, we have
$$\mathcal{K}^{\frac{p}{n+p+1}}\leq \left(C+C't^{-\frac{np}{(n+1)(p+1)}}\right),$$
where $C$ and $C'$ are constants depending on $n,p,R_{-}$ and $R_{+}.$
\end{lemma}
\begin{proof} Let $\alpha:=1-\frac{(n+2)p}{n+1+p}$ and $\beta:=-\frac{p}{n+1+p}.$  Consider the  function
$$\Psi=\frac{s^{\alpha}S_n^{\beta}}{s-R_{-}/2},$$ where $S_n$ stands for the $n$-th symmetric polynomial in the radii of curvature as a function on the sphere $\mathbb{S}^n$.
Using the maximum principle, we will show that $\Psi$ is bounded from above by a function of $n, p, R_{-}, R_{+},$ and time.
At the point where the maximum of $\Psi$ occurs, we have
$$0=\bar{\nabla}_i\Psi=\bar{\nabla}_i \left(\frac{s^{\alpha}S_n^{\beta}}{s-R_{-}/2}\right)\ \ \ {\hbox{and}}\ \ \  \bar{\nabla}_i\bar{\nabla}_j \Psi\leq 0.$$
Hence, we obtain
$$\frac{\bar{\nabla}_i (s^\alpha S_n^\beta)}{s-R_-/2}=\frac{(s^\alpha S_n^\beta) \bar{\nabla}_i s}{(s-R_-/2)^2}$$ and, consequently,
\begin{equation}\label{e: tso}
\bar{\nabla}_i\bar{\nabla}_j\left(s^{\alpha}S_n^{\beta}\right)+\bar{g}_{ij}\left(s^{\alpha}S_n^{\beta}\right)\leq
\frac{s^{\alpha}S_n^{\beta}\mathfrak{r}_{ij}-R_{-}/2s^{\alpha}S_n^{\beta}\bar{g}_{ij}}{s-R_{-}/2}.
\end{equation}
We calculate
\begin{align*}
\partial_t\Psi&=-\frac{\beta s^{\alpha}S_n^{\beta-1}}{s-R_{-}/2}(\dot{S}_n)_{ij}
\left[\bar{\nabla}_i\bar{\nabla}_j\left(s^{\alpha}S_n^{\beta}\right)+\bar{g}_{ij}\left(s^{\alpha}S_n^{\beta}\right)\right]\\
&+\frac{S_n^{\beta}}{s-R_{-}/2}\partial_t s^{\alpha}+\frac{s^{2\alpha}S_n^{2\beta}}{(s-R_{-}/2)^2},
\end{align*}
 where $\displaystyle (\dot{S}_n)_{ij} := \frac{\partial S_n}{\partial  \mathfrak{r}_{ij}}$ is the derivative of the $S_n$ with respect to the entry $ \mathfrak{r}_{ij}$ of the radii of curvature matrix. By Theorem 1, page 102 of \cite{Mi},
applied to the top symmetric polynomial, we have that $\displaystyle
(\dot{S}_n)_{ij}$ is a
positive definite bilinear form as long as $ \partial K$ has
positive Gauss curvature at all points. Notice that
\begin{equation}\label{e: one to last}
\frac{S_n^{\beta}}{s-R_{-}/2}\partial_t s^{\alpha}=-\alpha\Psi^2+\frac{\alpha R_{-}}{2}\frac{s^{2\alpha-1}S_n^{2\beta}}{(s-R_-/2)^2}\leq -\alpha\Psi^2.
\end{equation}
Thus, using inequalities (\ref{e: tso}) and (\ref{e: one to last}) we infer that, at the point where the maximum of $\Psi$ is reached, we have
\begin{equation}\label{e: last step tso}
\partial_t\Psi\leq\Psi^2\left(-n\beta-\alpha+1+\frac{\beta R_{-}}{2}\mathcal{H}\right).
\end{equation}
We can control the mean curvature $\mathcal{H}$ from below by a positive power of $\Psi.$ First notice that $\mathcal{H}\geq \frac{n}{S_n^{\frac{1}{n}}}.$ Therefore
\begin{equation*}
\mathcal{H}\geq n\left(\frac{s-R_{-}/2}{s^{\alpha}S_{n}^{\beta}}\right)^{\frac{1}{n\beta}}\left(\frac{s^{\alpha}}{s-R_{-}/2}\right)^{\frac{1}{n\beta}}
\geq n\Psi^{-\frac{1}{n\beta}} \left(\frac{R_{-}^{\alpha}}{R_{-}-R_{-}/2}\right)^{\frac{1}{n\beta}}.
\end{equation*}
Therefore, we can rewrite the inequality (\ref{e: last step tso}) as follows
\begin{align*}
\partial_t\Psi&\leq \Psi^2\left(-n\beta-\alpha+1+\frac{n\beta R_{-}}{2}\Psi^{-\frac{1}{n\beta}} \left(\frac{R_{-}^{\alpha}}{R_{-}/2}\right)^{\frac{1}{n\beta}}\right)\\
&=-\Psi^2\left(C'(n,p,R_{-}, R_+)\Psi^{\frac{n+1+p}{np}}-C(n,p)\right),
\end{align*} for positive constants $C(n,p)$ and $C'(n,p,C, R_{-}, R_+)$.
Hence,
$$\Psi\leq \max\left\{C(n,p,R_{-},R_{+}),C'(n,p,R_{-},R_+)t^{-\frac{np}{(n+1)(p+1)}}\right\}$$
for new constants $C$ and $C'$. The corresponding claim for $\mathcal{K}$ follows.
\end{proof}

 Pertaining to the flow by powers of the Gauss curvature, a powerful technique to obtain a uniform lower bound on the Gauss curvatures is using a Harnack's inequality and a lower displacement bound \cite{BA1}. The lower displacement bound controls how much the support of the evolving body decreases depending on time. The displacement bound is obtained by looking at how appropriate barriers, usually balls, with appropriate centers, move along the flow, combined with a containment principle. Despite the fact that we can prove a Harnack's estimate for the $p$-flow, one still needs to obtain the lower displacement bound via barriers. Arbitrarily centered balls are not good barriers in this case because the $p$-flow is not invariant under Euclidean translations and the $p$-flow acting on a ball depends on the choice of the origin. To overcome this difficulty we introduce a new technique. We look at the geometric flow that evolves the dual convex body, the dual $p$-flow.

Let $K^{\circ}$ denote the polar
body associated with $K$ with respect to the origin
$$
K^{\circ} = \{ y \in \mathbb{R}^{n+1} \mid x \cdot y \leq 1,\ \forall x
\in K \}.$$
We will use further the following lemma  proved in \cite{S}.
\begin{lemma}[The dual $p$-flow]\label{lem: dual p flow}\cite{S}
Let $\{K_t\}_{[0,T)}$ be a smooth, strictly convex solution of the evolution equation (\ref{e: p flow ev of s}). Then $\{K_t^{\circ}\}_{[0,T)}$ is a solution of the following evolution equation, the expanding $p$-flow (alternatively called the dual $p$-flow):
$$\partial_t s^{\circ}=s^{\circ}\left(\frac{\mathcal{K}^{\circ}}{s^{\circ n+2}}\right)^{-\frac{p}{n+1+p}}.$$
\end{lemma}

It is in next lemma that we need to restrict to the case $p<\frac{n+1}{n-1}.$

\begin{lemma}[Lower bound on the Gauss curvature]\label{lem: lower G}
Let $1\leq p<\frac{n+1}{n-1}.$ Assume that $\{K_t\}_{[0,t_0]}$ is a smooth, strictly convex solution of equation (\ref{e: p flow ev of s}) with $0<R_{-}\leq r_{-}(K_t)\leq r_{+}(K_t)\leq R_{+}< + \infty$ for $t\in[0,t_0]$. Then
$$\mathcal{K}^{\frac{p}{n+p+1}}\geq \frac{1}{C+C't^{\frac{np}{(n-1)p-(n+1)}}},$$
where $C$ and $C'$ are constants depending on $n,~p,~R_{-}$ and $R_{+}.$
\end{lemma}
\begin{proof} Recall from Lemma \ref{lem: dual p flow} that $\partial_t s^{\circ}=s^{\circ}\left(\frac{\mathcal{K}^{\circ}}{s^{\circ n+2}}\right)^{-\frac{p}{n+1+p}}.$ We define $\alpha:=1+\frac{(n+2)p}{n+1+p}$ and
$\beta:=\frac{p}{n+1+p}.$ Therefore the dual flow takes the following form $\partial_t s^{\circ}=s^{\circ\alpha}S_n^{\circ\beta}.$ Since $R_{-}\leq r_{-}(K_t)\leq r_{+}(K_t)\leq R_{+}$, we have
$$\frac{1}{R_{+}}\leq r_{-}(K^{\circ}_t)\leq r_{+}(K_t^{\circ})\leq \frac{1}{R_{-}}.$$
Define
$$R_{-}^{\circ}:=\frac{1}{R_{+}},~~R_{+}^{\circ}:=\frac{1}{R_{-}}$$
and consider the function
$$\Phi=\frac{s^{\circ\alpha}S_n^{\circ\beta}}{2R_{+}^{\circ}-s^{\circ}}.$$
The subsequent computation is carried out at the point where the minimum of $\Phi$ occurs:
$$0=\bar{\nabla}_i\Phi=\bar{\nabla}_i \left(\frac{s^{\circ\alpha}S_n^{\circ\beta}}{2R_{+}^{\circ}-s^{\circ}}\right)\ \ \ {\hbox{and}}\ \ \  \bar{\nabla}_i\bar{\nabla}_j \Phi\geq 0,$$
hence we obtain
$$ \frac{\bar{\nabla}_i (s^{\circ\alpha}S_n^{\circ\beta})}{2R_{+}^{\circ}-s^{\circ}}=-\frac{ s^{\circ\alpha}S_n^{\circ\beta} \bar{\nabla}_i s^\circ}{(2R_{+}^{\circ}-s^{\circ})^2} $$ and
\begin{equation}\label{e: tso dual}
\bar{\nabla}_i\bar{\nabla}_j\left(s^{\circ\alpha}S_n^{\circ\beta}\right)+\bar{g}_{ij}\left(s^{\circ\alpha}S_n^{\circ\beta}\right)\geq
\frac{-s^{\circ\alpha}S_n^{\circ\beta}\mathfrak{r}_{ij}^{\circ}+2R_{+}^{\circ}s^{\circ\alpha}S_n^{\circ\beta}\bar{g}_{ij}}{2R_{+}^{\circ}-s^{\circ}}.
\end{equation}
Calculating
\begin{align*}
\partial_t\Phi=&\frac{\beta s^{\circ\alpha}S_n^{\circ\beta-1}}{2R_{+}^{\circ}-s^{\circ}}(\dot{S}_n^{\circ})_{ij}
\left[\bar{\nabla}_i\bar{\nabla}_j\left(s^{\circ\alpha}S_n^{\circ\beta}\right)+\bar{g}_{ij}\left(s^{\circ\alpha}S_n^{\circ\beta}\right)\right]
+\frac{S_n^{\circ\beta}}{2R_{+}^{\circ}-s^{\circ}}\partial_t s^{\circ\alpha}\\
&+\frac{s^{\circ2\alpha}S_n^{\circ2\beta}}{(2R_{+}^{\circ}-s^{\circ})^2},
\end{align*}
and applying inequality (\ref{e: tso dual}), we conclude that
\begin{equation}\label{e: last step tso dual}
\partial_t\Phi\geq \Phi^2\left(1-n\beta-\alpha+2\beta R_{+}^{\circ}\mathcal{H}^{\circ}\right).
\end{equation}
We now estimate the mean curvature $\mathcal{H}^{\circ}$ from below by a negative power of $\Phi.$ As in the proof of the previous lemma, we have
\begin{equation*}
\mathcal{H}^{\circ}\geq n\left(\frac{2R_{+}^{\circ}-s^{\circ}}{s^{\circ\alpha}S_{n}^{\circ\beta}}\right)^{\frac{1}{n\beta}}\left(\frac{s^{\circ\alpha}}{2R_{+}^{\circ}-s^{\circ}}\right)^{\frac{1}{n\beta}}
\geq n\Phi^{-\frac{1}{n\beta}} \left(\frac{R_{-}^{\circ\alpha}}{2R_{+}^{\circ}-R_{-}^{\circ}}\right)^{\frac{1}{n\beta}}.
\end{equation*}
Consequently, inequality (\ref{e: last step tso dual}) can be rewritten as follows
\begin{align*}
\partial_t\Phi&\geq \Phi^2\left(1-n\beta-\alpha+2R^{\circ}_{+}n\beta\Phi^{-\frac{1}{n\beta}} \left(\frac{R_{-}^{\circ\alpha}}{2R_{+}^{\circ}-R_{-}^{\circ}}\right)^{\frac{1}{n\beta}}\right)\\
&=\Phi^2\left(-C(n,p)+C'(n,p,R_{-}^{\circ},R_{+}^{\circ})\Phi^{-\frac{n+1+p}{np}}\right),
\end{align*} for positive constants $C(n,p)$ and $C'(n,p,R_{-}^{\circ},R_{+}^{\circ})$. \newline
Hence
$$\partial_t\left(\frac{1}{\Phi}\right)\leq -C'(n,p,R_{-}^{\circ},R_{+}^{\circ})\left(\frac{1}{\Phi}\right)^{\frac{n+1+p}{np}}+C(n,p),$$
which implies
$$\frac{1}{\Phi}\leq \max\left\{C,C't^{\frac{np}{(n-1)p-(n+1)}}\right\}$$
for new constants $C$ and $C'$. Equivalently, we have a bound for $\Phi$ from below.

Therefore, we have bounded from above $\mathcal{K}^{\circ}$ in terms of $n,p,R_{-}, R_{+}$ and time.  To complete the proof we recall the following fact: for every $x \in \partial K$, there exists an $x^\circ \in \partial K^\circ$ such that
$$\left(\frac{\mathcal{K}}{s^{n+2}}\right)(x)\left(\frac{\mathcal{K}^{\circ}}{s^{\circ n+2}}\right)(x^{\circ})=1,$$
where $x$ and $x^{\circ}$ are related by $\langle x, x^{\circ}\rangle=1,$ with $\langle \, , \rangle$ the inner product in $\mathbb{R}^{n+1}$. A proof of this identity in the smooth setting is simple. A proof in a more general non-smooth setting can be found in \cite{H}, see also \cite{LR}.
By the above identity, we conclude that $\mathcal{K}$ is bounded from below by constants depending on $n,p,R_{-}, R_{+}$ and time.
\end{proof}
We point out that in concluding the long time existence of solutions, and asymptotic behavior of the flow in $\mathbb{R}^2$ the following two lemmas in this section are not necessary.

\begin{lemma}[Lower bound on the principal curvatures]\label{lem: lower P}
Assume that $n>1.$ Let $\{K_t\}_{[0,t_0]}$ be a smooth strictly convex solution  of equation (\ref{e: p flow ev of s}) with $0<R_{-}\leq r_{-}(K_t)\leq r_{+}(K_t)\leq R_{+}< +\infty$ and suppose that$$C_1\leq S_n\leq C_2$$
for all $t\in[0,t_0].$
Then there exist $C$ and $C'$ depending on $n,p,R_{-}$, $R_{+},C_1$ and $C_2$ such that
$$\frac{1}{\kappa_i}\leq \left(C+C't^{-(n-1)}\right),$$ for all $t\in[0,t_0].$
\end{lemma}
\begin{proof}We first compute the evolution equation of $\mathfrak{r}_{ij}=\bar{\nabla}_i\bar{\nabla}_js+s\bar{g}_{ij}$ for which a similar computation has been carried in \cite{BA1}.

Set $\alpha:=-1+\frac{(n+2)p}{n+1+p}$ and $\beta:=\frac{p}{n+1+p}.$
\begin{align*}
\partial_t\mathfrak{r}_{ij}&=\beta s^{-\alpha}S_n^{-(1+\beta)}(\dot{S}_n)_{kl}\bar{\nabla}_k\bar{\nabla}_l\mathfrak{r}_{ij}-\beta(\beta+1)s^{-\alpha}
S_n^{-(2+\beta)}\bar{\nabla}_iS_n\bar{\nabla}_jS_n\\
&+\beta s^{-\alpha}S_n^{-(1+\beta)}(\ddot{S}_n)_{kl;mn}\bar{\nabla}_i\mathfrak{r}_{kl}\bar{\nabla}_j\mathfrak{r}_{mn}\\
&+(n\beta-1)s^{-\alpha}S_n^{-\beta}\bar{g}_{ij}-\beta s^{-\alpha}S_n^{-(1+\beta)}(\dot{S}_n)_{kl}\bar{g}_{kl}\mathfrak{r}_{ij}\\
&-S_n^{-\beta}\bar{\nabla}_i\bar{\nabla}_j s^{-\alpha}+\beta S_n^{-(1+\beta)}\bar{\nabla}_is^{-\alpha}\bar{\nabla}_j S_n+\beta S_n^{-(1+\beta)}\bar{\nabla}_js^{-\alpha}\bar{\nabla}_i S_n.
\end{align*}
\begin{description}
  \item[a] \emph{Estimating the terms on the first line:} The first term on the first line is an essential good term viewed as an elliptic operator which is non-positive at the point and direction where the maximum of $\mathfrak{r}_{ij}$ occurs. The second term is an essential good negative term.
  \item[b] \emph{Estimating the term on the second line:} Concavity of $S_{n}^{\frac{1}{n}}$, see again \cite{Mi}, gives
  \begin{equation}\label{ie: second derv}
  \left[(\ddot{S}_n)_{kl;mn}-\frac{n-1}{n{S}_n}(\dot{S}_n)_{kl}(\dot{S}_n)_{mn}\right]
      \bar{\nabla}_{i}\mathfrak{r}_{kl}\bar{\nabla}_j\mathfrak{r}_{mn}\leq0.
  \end{equation}
  \item[c] \emph{Estimating the terms on the last line:}
 \begin{align*}
 \bar{\nabla}_i\bar{\nabla}_js^{-\alpha}&=-\alpha \frac{\bar{\nabla}_i\bar{\nabla}_js}{s^{\alpha+1}}+\alpha(\alpha+1)
  \frac{ \bar{\nabla}_is\bar{\nabla}_js}{s^{\alpha+2}}\\
  &=-\alpha \frac{(\mathfrak{r}_{ij}-\bar{g}_{ij}s)}{s^{\alpha+1}}+\alpha(\alpha+1)
  \frac{ \bar{\nabla}_is \bar{\nabla}_js}{s^{\alpha+2}}.
  \end{align*}
  This gives
  \begin{equation}\label{e: line four}
  -S_n^{-\beta}\bar{\nabla}_i\bar{\nabla}_js^{-\alpha}\leq  CS_n^{-\beta}\mathfrak{r}_{ij}+CS_n^{-\beta}\bar{g}_{ij},
  \end{equation}
  where we used boundedness of $\bar{\nabla}_i s$ from above and the assumptions of the lemma. Notice that $|x|^2=s^2+|\bar{\nabla}s|^2$. Therefore, as $s$ is bounded, $|\bar{\nabla}s|$ must also be bounded. Here we used $| \, \cdot \, |$ for the Euclidean norm in $\mathbb{R}^{n+1}.$
  The other term on the last line can be estimated by Young's inequality:
  \begin{equation}\label{e: line four, the other}
  |\bar{\nabla}_is^{-\alpha}\bar{\nabla}_jS_n|\leq C\varepsilon |\bar{\nabla}_j S_n|^2+C\varepsilon^{-1}.
  \end{equation}
\end{description}
Combining inequality (\ref{ie: second derv}) and inequalities (\ref{e: line four}) and (\ref{e: line four, the other}), for $\varepsilon$ small enough, we have
$$\partial_t\mathfrak{r}_{ij}\leq \beta s^{-\alpha}S_n^{-(1+\beta)}(\dot{S}_n)_{kl}\bar{\nabla}_k\bar{\nabla}_l\mathfrak{r}_{ij}+CS_n^{-\beta}(\mathfrak{r}_{ij}+\bar{g}_{ij})-\beta S_n^{-(1+\beta)}(\dot{S}_n)_{kl}\bar{g}_{kl}\mathfrak{r}_{ij}.$$
Therefore, the maximum of the hypersurface's $\mathfrak{r}_{ij}$, as a function of time, satisfies
$$\partial_t(\mathfrak{r}_{ij})_{\max}\leq CS_n^{-\beta}\left(\bar{g}_{ij}+(\mathfrak{r}_{ij})_{\max}-C'S_n^{-\frac{1}{n-1}}(\mathfrak{r}_{ij})_{\max}^{\frac{n}{n-1}}\right).$$
This implies that, for $(\mathfrak{r}_{ij})_{\max}$ very large, the quantity in parentheses is negative, while $S_n$ is bounded away from zero, hence the behavior of  $(\mathfrak{r}_{ij})_{\max}$ when large is modeled by the differential inequality $$\frac{d (\mathfrak{r}_{ij})_{\max}}{dt} \leq -C (\mathfrak{r}_{ij})_{\max}^{\frac{n}{n-1}}.$$ We thus conclude that
$$\max_{i,j}(\mathfrak{r}_{ij})_{\max}\leq (C+C't^{-(n-1)}),$$
for some positive constants $C,C'$.
As, for any real symmetric matrix $A$, its highest eigenvalue is $\displaystyle \lambda_{max}(A)= \sup_{u \in \mathbb{R}^n,\ \|u \|=1}|\langle u, Au \rangle |$, we obtain the upper bound on the highest radius of curvature of the form $C+C't^{-(n-1)}$.
\end{proof}
\begin{lemma}[Lower and upper bounds on the principal curvatures]\label{lem: lower and upper P}
Assume that $n>1.$ Let $\{K_t\}_{[0,t_0]}$ be a smooth, strictly convex solution of equation (\ref{e: p flow ev of s}) with $0<R_{-}\leq r_{-}(K_t)\leq r_{+}(K_t)\leq R_{+}< + \infty$ and
$$C_1\leq S_n\leq C_2$$
for all $t\in[0,t_0].$
Then there exist constants $C_3$ and $C_4$ depending on $n,p,R_{-},$ $R_{+},C_1$ and $C_2$, such that $\forall t\in[0,t_0]$
$$\frac{1}{C_3(1+t^{-(n-1)})}\leq\kappa_i\leq C_4\left(1+t^{-(n-1)}\right)^{n-1}.$$
\end{lemma}
\begin{proof} The lower bound on the principal curvatures has been, in fact, established in Lemma \ref{lem: lower P}. Consequently, we also obtain now the upper bound as the product of the principal curvatures is bounded from above. Suppose that $\kappa_1 \geq \kappa_2 \geq \ldots \geq \kappa_n$, then $$C_1^{-1} \geq \mathcal{K}= \Pi_{i=1}^n \kappa_i = \kappa_1 \cdot \Pi_{i=1}^{n-1} \kappa_i \geq \kappa_1 (C_3+C_3t^{-(n-1)})^{-(n-1)}.$$
\end{proof}
\begin{theorem}\label{thm: zero volume}
Let $1\leq p<\frac{n+1}{n-1}$ be a real number. Let $x_{K_0}:\mathbb{S}^n\to\mathbb{R}^{n+1}$ be a smooth, strictly convex embedding of $K_0\in\mathcal{K}_{sym}.$ Then there exists a unique solution $x:\mathbb{S}^n\times [0,T)\to\mathbb{R}^{n+1}$ of equation (\ref{e: p flow ev of x}) with initial data $x_{K_0}$, for a maximal finite $T$, such that $\lim\limits_{t\to T}V(K_t)=0.$
\end{theorem}
\begin{proof}
Suppose that $S_0$ is a sphere which, at time zero, encloses $K_0$. It is clear that, by applying the $p$-flow to $S_0$, the evolving spheres $S_t$ converge to a point in finite time. By a comparison principle, $ K_t$ remains in the closure of $S_t$, therefore $T$ must be finite. Suppose now that $V(K_t)$ does not tend to zero. Then, we must have $s\geq r$, for some $r>0$ on $[0,T)$. By Lemmas \ref{lem: upper G}, \ref{lem: lower G}, \ref{lem: lower P} and \ref{lem: lower and upper P} the principal curvatures of the solution remains uniformly bounded on $[0,T)$ from below and above. Consequently the evolution equation (\ref{e: p flow ev of x}) is uniformly parabolic on $[0,T)$, and bounds on higher derivatives of the support function follows by \cite{K} and Schauder theory. Hence, we can extend the solution after time $T$,  contradicting its definition.
\end{proof}

\section{Convergence of the volume preserving $p$-flow}
In this section, we will conclude the proof of the main theorem.

We will start with a few facts on convex bodies.
For a given convex body $K$, we call the
volume of $K$ its Lebesgue measure as a subset of $\mathbb{R}^{n+1}$.
A celebrated affine invariant quantity associated with $K$ is its affine surface area which for a body with a smooth boundary can be expressed by
\begin{equation}
 \Omega (K) = \int_{\mathbb{S}^n} {\cu}^{-\frac{n+1}{n+2}}\, d\mu_{ \mathbb{S}^n}.
\end{equation} The affine surface area satisfies an isoperimetric inequality which we will state soon in a more general setting.
The extension of the affine surface area defined by Lutwak
\cite{Lutwak2} for $p>1$ in the context of the Firey-Brunn-Minkowski theory of
convex bodies, called the $L_p$-affine or $p$-affine surface area, is
\begin{equation}
\Omega_p (K) =\int\limits_{\mathbb{S}^{n}}\frac{s}{\mathcal{K}}\left(\frac{\mathcal{K}}{s^{n+2}}\right)^{\frac{p}{n+1+p}}
  d\mu_{\mathbb{S}^{n}},
\label{eq:p_affine_surface}
\end{equation}
which reduces to the usual affine surface area for $p=1$.  At the core of the
centro-affine geometry, lies the $p$-affine isoperimetric inequality due to Lutwak, $p>1$
 which generalizes the classical $p=1$ case,
 \begin{equation}
\frac{\Omega_p^{n+p+1}(K)}{ V^{n-p+1} (K)}\leq (n+1)^{n+p+1}\omega_{n+1}^{2p},
 \end{equation}
\noindent where $\omega_{n+1}=V(B^{n+1})$ is the volume of the unit ball in $\mathbb{R}^{n+1},$
\cite{Lutwak2}. The equality holds if and only if $K$ is a centered ellipsoid.

The following result follows directly from the inequality of Proposition 4.2 in \cite{S}.
\begin{thmA}[Monotonicity of $p$-affine isoperimetric ratio]\label{thm: monoton}
Let $\{K_t\}_{[0,T)}$ be a smooth, strictly convex solution of equation (\ref{e: p flow ev of s}). Then the $p$-affine isoperimetric ratio, $\displaystyle\frac{\Omega_{p}^{n+1+p}(K_t)}{V^{n+1-p}(K_t)}$, is non-decreasing along the $p$-flow. The monotonicity is strict unless $K_t$ is an ellipsoid centered at the origin.
\end{thmA}
Finally, we recall  also the following famous theorem by Fritz John.
\begin{thmA}[John's Lemma]\cite{J}
 Let $K$ be a convex body in $\mathbb{R}^{n+1}.$ Then there exist absolute constants $c$ and $C$, and an affine linear transformation $L$ such that
$$c\leq \left(\frac{V(B^{n+1})}{V(K)}\right)^{\frac{1}{n+1}}s_{LK}\leq C,$$ where $B^{n+1}$ denotes the unit ball in  $\mathbb{R}^{n+1}$.
\end{thmA}

Let now $x:\mathbb{S}^n\times[0,T]\to\mathbb{R}^{n+1}$ be a solution of equation (\ref{e: p flow ev of x}). Then for each $\lambda>0$, note that $x_{\lambda}$  defined by $x_{\lambda}:\mathbb{S}^n\times\left[0,\lambda^{\frac{(2n+2)p}{n+1+p}}T\right]\to\mathbb{R}^{n+1}$  with
$$x_{\lambda}(\theta,t)=\lambda x\left(\theta,\lambda^{-\frac{(2n+2)p}{n+1+p}}t\right)$$
is also a solution of evolution equation (\ref{e: p flow ev of x}).

\noindent \textbf{Proof of the main theorem}: We follow the procedure in \cite{BA1}. Fix $t\in[0,T)$. Define $\tilde{s}$ a solution of (\ref{e: p flow ev of s}), by the rescaling property, as follows
$$\tilde{s}(z,\tau)=\left(\frac{V(B^{n+1})}{V(K_t)}\right)^{\frac{1}{n+1}}s\left(z,
t+\left(\frac{V(B^{n+1})}{V(K_t)}\right)^{-\frac{2p}{n+1+p}}\tau\right),$$
where $\tilde s(\cdot, 0)$ is the support function of $L_t\tilde{K}_t$
and $L_t\in SL(n+1)$ is obtained from John's Lemma for the convex body $K_t$. Therefore,
$$c\leq\tilde{s}(z,0)\leq C.$$
Let $B_r$ denote the ball of radius $r$ centered at the origin. Thus, $B_c$ is contained in the convex body associated with the support function $\tilde{s}(\cdot,0).$
The containment principle, see for example Proposition 2.2 in \cite{S}, insures that $B_{c/2}$ will be contained in the convex body associated with the support function $\tilde{s}(\cdot,\tau),$ for $\tau\in[0,\delta],$
where $\delta$ is the time that $B_c$ becomes $B_{c/2}$ under the $p$-flow. This time can be found explicitly as the evolution of a ball of radius $\rho$ centered at the origin is $\rho_t=-\rho^{(n+1-(2n+1)p)/(n+p+1)}$. Now Lemmas \ref{lem: upper G}, \ref{lem: lower G}, \ref{lem: lower P} and \ref{lem: lower and upper P} imply that there are uniform lower and upper bounds on the principal curvatures and on the speed of the flow on the time interval $[\delta/2,\delta].$ Therefore, by \cite{K}, we conclude that there are uniform bounds on higher derivatives of the curvature. Consequently, all quantities related to the original solution that are both scaling invariant and invariant under $SL(n+1)$ satisfy uniform bounds on the time interval $\left[t+\frac{C}{2}V(K_t)^{\frac{2p}{n+1+p}},t+CV(K_t)^{\frac{2p}{n+1+p}}\right].$ Since $t$ is arbitrary and $C$ is an absolute constant, we have uniform bounds on the time interval $[T/2,T).$ It means that all $SL(n+1)$ invariant quantities of the normalized solution to the $p$-flow are uniformly bounded on the time interval $[T/2,T).$ We point out here that if $n=1$, only Lemma \ref{lem: upper G} and Lemma \ref{lem: lower G} are needed to derive such uniform bounds on the time interval $[T/2,T)$.

Consequently, there is a sequence of times $\{t_k\}_{k\in\mathbb{N}}$ and $L_{t_k}\in SL(n+1)$, such that $t_k$ approaches $T$ and $L_{t_k}\tilde{K}_{t_k}$ converges in the $C^{\infty}$ topology to a convex body $\tilde{K}_{T}.$
Now monotonicity of the $p$-affine isoperimetric ratio and Theorem \ref{thm: zero volume} with a similar argument as in \cite{BA1}, implies that $\tilde{K}_T$ must be an ellipsoid. Therefore,
$$\lim_{t_k\to T}\frac{\Omega_{p}^{n+1+p}(K_{t_k})}{V^{n+1-p}(K_{t_k})}=(n+1)^{n+1+p}\omega_{n+1}^{2p},$$
and again by monotonicity of the $p$-affine isoperimetric ratio
$$\lim_{t\to T}\frac{\Omega_{p}^{n+1+p}(K_t)}{V^{n+1-p}(K_t)}=(n+1)^{n+1+p}\omega_{n+1}^{2p}.$$
From the equality case in the $p$-affine isoperimetric inequality \cite{Lutwak2}, it follows that, modulo $SL(n+1)$,
$$\lim_{t\to T}\left(\frac{V(B^{n+1})}{V(K_t)}\right)^{\frac{1}{n+1}}K_t=B^{n+1}$$
 sequentially in the $C^{\infty}$ topology. The proof is complete.\newline

\textbf{Acknowledgment.}\\
 We would like to thank the referees for their helpful comments and the thorough attention given to our manuscript.
\bibliographystyle{amsplain}

\end{document}